\documentclass[12pt]{amsart}

\usepackage[margin=1.25in]{geometry}
\usepackage{graphicx}
\usepackage{amsrefs}
\usepackage{xcolor}
\newcommand{\arxiv}[1]{arXiv:#1}

\newcommand{\new}[1]{{#1}}
\newcommand{\old}[1]{{#1}}

\newtheorem{theorem}{Theorem}
\newtheorem{lemma}[theorem]{Lemma}

\newtheorem{proposition}[theorem]{Proposition}
\theoremstyle{definition}
\newtheorem{remark}[theorem]{Remark}
\newtheorem{definition}[theorem]{Definition}
\newtheorem{problem}[theorem]{\new{Open} Problem}

\newcommand{\eref}[1]{(\ref{e.#1})}
\newcommand{\fref}[1]{Figure \ref{f.#1}}
\newcommand{\pref}[1]{Proposition \ref{p.#1}}
\newcommand{\tref}[1]{Theorem \ref{t.#1}}
\newcommand{\lref}[1]{Lemma \ref{l.#1}}
\newcommand{\qref}[1]{Problem \ref{q.#1}}

\newcommand{\R}{\mathbb{R}}
\newcommand{\Z}{\mathbb{Z}}
\newcommand{\C}{\mathbb{C}}
\newcommand{\N}{\mathbb{N}}
\newcommand{\trace}{\operatorname{tr}}
\newcommand{\ep}{\varepsilon}
\newcommand{\mat}[1]{{\left[ \begin{matrix} #1 \end{matrix} \right]}}
\newcommand{\smat}[1]{{\left[ \begin{smallmatrix} #1 \end{smallmatrix} \right]}}

\begin{document}

\title{Stability of patterns in the Abelian sandpile}

\author{Wesley Pegden}
\address{Department of Mathematics, Carnegie Mellon University, Pittsburgh, PA}
\email{wes@math.cmu.edu}

\author{Charles K Smart}
\address{Department of Mathematics, The University of Chicago, Chicago, IL}
\email{smart@math.uchicago.edu}

\begin{abstract}
  We show that the patterns in the Abelian sandpile are stable.  The proof combines the structure theory for the patterns with the regularity machinery for non-divergence form elliptic equations.  The stability results allows one to improve weak-$*$ convergence of the Abelian sandpile to pattern convergence for certain classes of solutions.
\end{abstract}

\maketitle

\section{Introduction}

Consider the following discrete boundary value problem for a bounded open set $\Omega \subseteq \R^2$ with the exterior ball condition.  For each integer $n > 0$, let $u_n : \Z^2 \to \Z$ be the point-wise least function that satisfies
\begin{equation}
  \label{e.fde}
  \begin{cases}
    \Delta u_n \leq 2 & \mbox{in } \Z^2 \cap n \Omega \\
    u_n \geq 0 & \mbox{in } \Z^2 \setminus n \Omega,
  \end{cases}
\end{equation}
where $\Delta u(x) = \sum_{y \sim x} (u(y) - u(x))$ is the Laplacian on $\Z^2$.  Were it not for the integer constraint on the range of $u_n$, this would be the standard finite difference approximation of the Poisson problem on $\Omega$.  The integer constraint imposes a non-linear structure that drastically changes the scaling limit.  In particular, the Laplacian $\Delta u_n$ is not constant in general.  For example, in the case where $\Omega$ is a unit square, depicted in \fref{identity}, $\Delta u_n$ a fractal structure reminiscent of a Sierpinski gasket.  Upon close inspection one finds that the triangular regions of this image, displayed in more detail in \fref{zoom}, are filled by periodic patterns.

\begin{figure}[h]
  \includegraphics[width=.23\textwidth]{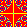}
  \hfill
  \includegraphics[width=.23\textwidth]{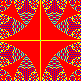}
  \hfill
  \includegraphics[width=.23\textwidth]{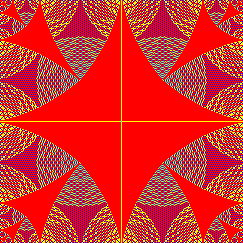}
  \hfill
  \includegraphics[width=.23\textwidth]{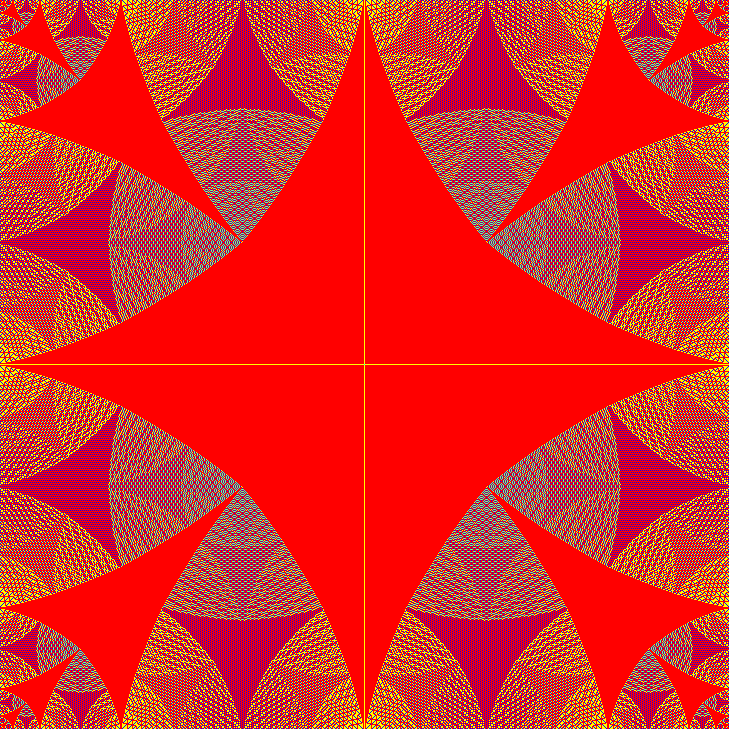}
  \caption{$\Delta u_n$ for $\Omega = (0,1)^2$ and $n = 3^3, 3^4, 3^5, 3^6$.  The colors blue, cyan, yellow, red correspond to values -1,0,1, 2.}
  \label{f.identity}
\end{figure}

The functions $u_n$ arise naturally as toppling functions in the Abelian sandpile model.  Recall that given a configuration of chips on $\Z^2$, \emph{toppling} a vertex distributes one chip from a vertex to each of its neighbors, while untoppling is the reverse operation.  The solution in \eqref{e.fde} is thus the minimum (i.e., maximally negative) sequence of topplings inside $\Z^2\cap n\Omega$ which does not result in greater than 2 chips at any site.  In particular $\Delta u_n+1$ is the unique \emph{recurrent} configuration on $\Z^2\cap n\Omega$ equivalent by topplings to the all-1's configuration in the square-lattice sandpile dynamics with toppling cutoff 3 (see Section \ref{s.prelim}).

We know from \cite{Pegden-Smart} the quadratic rescalings
\begin{equation*}
  \bar u_n(x) = n^{-2} u_n([ n x])
\end{equation*}
converge uniformly as $n \to \infty$ to the solution of a certain partial differential equation.  As a corollary, the rescaled Laplacians $\bar s_n(x) = \Delta u_n([n x])$ converge weakly-$*$ in $L^\infty(\Omega)$ as $n \to \infty$.  That is, the average of $\bar s_n$ over any fixed ball converges as $n \to \infty$.  The arguments establishing this are relatively soft and apply in great generality.  In this article we describe how, when $\bar u$ is sufficiently regular, the convergence of the $\bar s_n$ can be improved.

To get an idea of what we aim to prove, consider \fref{zoom}, which displays the triangular patches of \fref{identity} in greater detail.  It appears that, once a patch is formed, it is filled by a double periodic pattern, possibly with low dimensional defects.  This phenomenon has been known experimentally since at least the works of Ostojic \cite{Ostojic} and Dhar-Sadhu-Chandra \cite{Dhar-Sadhu-Chandra}.  The recent work of Kalinin-Shkolnikov \cite{Kalinin-Shkolnikov} identifies the defects, in a more restricted context, as tropical curves.

The shapes of the limiting patches are known in many cases.  Exact solutions for some other choices of domain are constructed by Levine and the authors \cite{Levine-Pegden-Smart-1}; the key point is that the notion of convergence used in this previous work ignores small-scale structure, and thus does not address the appearance of patterns.  The ansatz of Sportiello \cite{Sportiello} can be used to adapt these methods to the square with cutoff 3, which yields the continuum limit of the sandpile identity on the square.  Meanwhile, work of Levine and the authors \cite{Levine-Pegden-Smart-1} did classify the patterns which should appear in the sandpile, in the course of characterizing the structure of the continuum limit of the sandpile.  To establish that the patterns themselves appear in the sandpile process, it remains to show that this pattern classification is exhaustive, and that the patterns actually appear where they are supposed to.  In this manuscript  we complete this framework, and our results allow one to prove that the triangular patches are indeed composed of periodic patterns, up to defects whose size we can control.

\begin{figure}[h]
  \includegraphics[width=\textwidth]{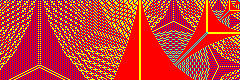}
  \caption{The patterns in the upper right corner of $\Delta u_n$ for $\Omega = (0,1)^2$ and $n = 600$.  The colors blue, cyan, yellow, red correspond to value $-1$, $0$, $1$, $2$.}
  \label{f.zoom}
\end{figure}

We describe our result for \eref{fde} with $\Omega = (0,1)^2$, leaving the more general results for later.  A doubly periodic {\em pattern} $p : \Z^2 \to \Z$ is said to $R$-match an {\em image} $s : \Z^2 \to \Z$ at $x \in \Z^2$ if, for some $y \in \Z^2$,
\begin{equation*}
  s(x+z) = p(y+z) \quad \mbox{for } z \in \Z^2 \cap B_R,
\end{equation*}
where $B_R$ is the Euclidean ball of radius $R$ and center $0$.

\begin{theorem}
  \label{t.main}
  Suppose $\Omega = (0,1)^2$.  There are disjoint open sets $\Omega_k \subseteq \Omega$ and doubly periodic patterns $p_k : \Z^2 \to \Z$ for each $k\geq 1$, and constants $L > 1$ and \new{ $\alpha,\delta \geq 0$} such that the following hold for all $n\geq 1$:
  \begin{enumerate}
  \item $|\Omega \setminus \cup_{1 \leq k \leq n} \Omega_k| \leq \new{ n^{-\delta}}$.
  \item For all $1 < r < n$, the pattern $p_k$ $r$-matches the image $\Delta u_n$ at a \new{$1 - L^k n^{-\alpha/4} r^{1/2}$} fraction of points in $n \Omega_k$.
  \end{enumerate}
\end{theorem}

We expect that the exponents in this theorem, while effective, are suboptimal.  In simulations, the pattern defects appear to be one dimensional.  This leads to the following problem.

\begin{problem}
  Improve the above estimate to a $1 - L^k n^{-1} r$ fraction of points.
\end{problem}

In fact, we expect that pattern convergence can be further improved in certain settings.  We see below that the points of the triangular patches in the continuum limit of \eref{fde} are all triadic rationals.  Moreover, when we select $n = 3^m$, then the patterns appear without any defects, as in \fref{identity}.  We expect this is not a coincidence.  These so-called ``perfect Sierpinski'' sandpiles have been investigated by Sportiello \cite{Sportiello} and appear in many experiments \new{ \cite{sadhu2010pattern,dhar2009pattern,Caracciolo-Paoletti-Sportiello,Dhar-Sadhu-Chandra}}.

\begin{problem}
  \label{q.perfect}
  Show that, when $n$ is a power of three, the patterns in patches larger than a constant size have no defects.  \new{ Let's discuss this wording.}
\end{problem}

Our proof has three main ingredients.  First, we prove that the patterns in the Abelian sandpile are in some sense stable.  This is a consequence of the classification theorem for the patterns and the growth lemma for elliptic equations in non-divergence form.  Second, we obtain a rate of convergence to the scaling limit of the Abelian sandpile when the limit enjoys some additional regularity.  This is essentially a consequence of the Alexandroff-Bakelman-Pucci estimate for uniformly elliptic equations.  Third, the limit of \eref{fde} when $\Omega = (0,1)^2$ has a piece-wise quadratic solution that can be explicitly computed by our earlier work.  The combination of these three ingredients implies that the patterns appear as \fref{identity} suggests.

The Matlab/Octave code used to compute the figures for this article is included in the arXiv upload and may be freely used and modified.

\subsection*{Acknowledgments}  Both authors are partial supported by the National Science Foundation and the Sloan Foundation.  The second author wishes to thank Alden Research Laboratory in Holden, Massachusetts, for their hospitality while some of this work was completed.

\section{Preliminaries}
\label{s.prelim}

\subsection{Recurrent functions} We recall the notion of being a locally least solution of the inequality in \eref{fde}.

\begin{definition}
  A function $v : \Z^2 \to \Z$ is recurrent in $X \subseteq \Z^2$ if $\Delta v \leq 2$ in $X$ and
  \begin{equation*}
    \sup_Y (v - w) \leq \sup_{X \setminus Y} (v - w)
  \end{equation*}
  holds whenever $w : \Z^2 \to \Z$ satisfies $\Delta w \leq 2$ in a finite $Y \subseteq X$.
\end{definition}

With this terminology, $u_n$ is characterized by being recurrent in $\Z^2 \cap n \Omega$ and zero outside.  The word recurrent usually refers to a condition on configurations $s : X \to \N$ in the sandpile literature \cite{Levine-Propp}.  These notions are equivalent for configurations of the form $s = \Delta v$.  That is, $v$ is a recurrent function if and only if $\Delta v$ is a recurrent configuration.

\subsection{Scaling limit}

We recall that the scaling limit of the Abelian sandpile.

\begin{proposition}[\cite{Pegden-Smart}]
  The rescaled solutions $\bar u_n$ of \eref{fde} converge uniformly to the unique solution $\bar u \in C(\R^2)$ of
  \begin{equation}
    \label{e.pde}
    \begin{cases}
      D^2 \bar u \in \partial \Gamma & \mbox{in } \Omega \\
      \bar u = 0 & \mbox{on } \R^2 \setminus \Omega,
    \end{cases}
  \end{equation}
  where $\Gamma \subseteq \R^{2 \times 2}_{sym}$ is the set of $2 \times 2$ real symmetric matrix $A$ for which there is a function $v : \Z^2 \to \Z$ satisfying
  \begin{equation}
    \label{e.odometer}
    v(x) \geq \tfrac12 x \cdot A x + o(|x|^2) \quad \mbox{and} \quad \Delta v(x) \leq 2 \quad \mbox{for all } x \in \Z^2,
  \end{equation}
  \new{ and $\partial \Gamma$ denotes the (topological) boundary of the set $\Gamma\subseteq \R^{2 \times 2}_{sym}$.}
\end{proposition}

The partial differential equation \eref{pde} is interpreted in the sense of viscosity.  This means that if a smooth test function $\varphi \in C^\infty(\Omega)$ touches $\bar u$ from below or above at $x \in \Omega$, then the Hessian $D^2 \varphi(x)$ lies in $\Gamma$ or the closure of its complement, respectively.   That this makes sense follows from standard viscosity solution theory, see for example \cite{Crandall}, and the following basic properties of the set $\Gamma$.

\begin{proposition}[\cite{Pegden-Smart}]
  The following holds for all $A, B \in \R^{2 \times 2}_{sym}$.
  \begin{enumerate}
  \item $A \in \Gamma$ implies $\trace A \leq 2$.
  \item $\trace A \leq 1$ implies $A \in \Gamma$.
  \item $A \in \Gamma$ and $B \leq A$ implies $B \in \Gamma$.
  \end{enumerate}
\end{proposition}

These basic properties tell us, among other things, that the differential inclusion \eref{pde} is degenerate elliptic and that any solution $\bar u$ satisfies the bounds
\begin{equation*}
  1 \leq \trace D^2 \bar u \leq 2
\end{equation*}
in the sense of viscosity.  This implies enough a prior regularity that we have a unique solution $\bar u \in C^{1,\alpha}(\Omega)$ for all $\alpha \in (0,1)$.

\subsection{Notation}

Our results make use of several arbitrary constants which we do not bother to determine.  We number these according to the result in which they are defined; e.g., the constant $C_{\ref{p.structure}}$ is defined in Proposition \ref{p.structure}.   In proofs, we allow Hardy notation for constants, which means that the letter $C$ denotes a positive universal constant that may differ in each instance.  We let $D \varphi : \R^d \to \R^d$ and $D^2 \varphi : \R^d \to \R^{d \times d}_{sym}$ denote the gradient and hessian of a function $\varphi \in C^2(\R^d)$.  We let $|x|$ denote the $\ell^2$ norm of a vector $x \in \R^d$ and $|A|$ denote the $\ell^2$ operator norm of a matrix $A \in \R^{d \times e}$.

\subsection{Pattern classification}

The main theorem from \cite{Levine-Pegden-Smart-2} states that $\Gamma$ is the closure of its extremal points and that the set of extremal points has a special structure.  We recall the ingredients that we need.

\begin{proposition}[\cite{Levine-Pegden-Smart-2}]
  If
  \begin{equation*}
    \Gamma^+ = \{ P \in \Gamma : \mbox{there is an $\ep > 0$ such that $P - \ep I \leq B \in \Gamma$ implies $B \leq P$} \},
  \end{equation*}
  then $A \in \Gamma$ if and only if $A = \lim_{n \to \infty} A_n$ for some $A_n \leq P_n \in \Gamma^+$.
\end{proposition}

For each $P \in \Gamma^+$, there is a recurrent $o : \Z^2 \to \Z$ witnessing $P \in \Gamma$.  These functions $o$, henceforth called odometers, enjoy a number of special properties.  The most important for us is that the Laplacians $\Delta o$ are doubly periodic with nice structure.  Some of the patterns are on display in \fref{patterns}.  We exploit the structure of these patterns to prove stability.

\begin{figure}[h]
  \includegraphics[angle=90,width=.23\textwidth]{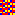}
  \hfill
  \includegraphics[angle=90,width=.23\textwidth]{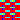}
  \hfill
  \includegraphics[angle=90,width=.23\textwidth]{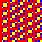}
  \hfill
  \includegraphics[angle=90,width=.23\textwidth]{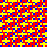}
  \caption{The Laplacian $\Delta o$ for several $P \in \Gamma^+$.  The colors blue, cyan, yellow, and red correspond to values $-1,0,1,2$.}
  \label{f.patterns}
\end{figure}

\begin{proposition}[\cite{Levine-Pegden-Smart-2}]
  \label{p.structure}
  There is a universal constant $C_{\ref{p.structure}} > 0$ such that, for each $P \in \Gamma^+$, there are $A, V \in \Z^{2 \times 3}$, $T \subseteq \Z^2$, and a function $o : \Z^2 \to \Z$, henceforth called an odometer function, such that the following hold.
  \begin{enumerate}
  \item $PV = A$,
  \item $1 \leq |V|^2 \leq C_{\ref{p.structure}} \det(V)$, where $|V|$ is the $\ell^2$ operator norm of $V$,
  \item $A^t Q V + V^t Q A = Q'$, where $Q = \smat{0 & 1 \\ -1 & 0}$ and $Q' = \smat{0 & 1 & -1 \\ -1 & 0 & 1 \\ 1 & -1 & 0}$,
  \item $A \smat{1\\1\\1} = V \smat{1\\1\\1} = \smat{0\\0}$,
  \item $o$ is recurrent and there is a quadratic polynomial $q$ such that $D^2 q = P$, $o - q$ is $V \Z^3$-periodic, and $|o - q| \leq C_{\ref{p.structure}} |V|^2$,
  \item $\Delta o = 2$ on $\partial T = \{ x \in T : y \sim x \mbox{ for some } y \in \Z^2 \setminus T \}$.
  \item If $x \sim y \in \Z^2$, then there is $z \in \Z^3$ such that $x, y \in T + V z$.
  \item If $z, w \in \Z^3$ and $(T + V z) \cap (T + Vw) \neq \emptyset$ if and only if $|z - w|_1 \leq 1$.
  \end{enumerate}
\end{proposition}

This proposition implies that $\Delta o$ is $V \Z^3$-periodic and that the set $\{ \Delta o = 2 \}$ has a unique infinite connected component.  Moreover, there is a fundamental tile $T \subseteq \Z^2$ whose boundary is contained in $\{ \Delta o = 2 \}$ and whose $V \Z^3$-translations cover $\Z^2$ with overlap exactly on the boundaries.  This structure is apparent in the examples in \fref{patterns}.

\subsection{Toppling cutoff}
\label{s.cutoff}

In \eref{fde} we've used the bound $2$ on the right-hand side.  In the language of sandpile dynamics, this means that sites topple whenever there are three or more particles at a vertex.  We have also used this bound in the literature review above, although the cited papers state their theorems with the bounds 3 or 1;  in particular, the paper \cite{Levine-Pegden-Smart-2} uses the bound $1$ for its results.  (In fact, the published version of the paper \cite{Levine-Pegden-Smart-2} is inconsistent in its use of the cutoff, so that in a few places, a value of 2 appears where 0 would be correct; this inconsistency has been corrected in the arXiv version of the paper.)  Translation between the conventions is perfomed by observing that the quadratic polynomial
\begin{equation*}
  q(x) = \tfrac12 x_1 (x_1 + 1)
\end{equation*}
is integer-valued on $\Z^2$, satisfies $\Delta q \equiv 1$, and has hessian $D^2 q \equiv \smat{1 & 0 \\ 0 & 0}$.  Since, for any $\alpha \in \Z$, we have $\Delta (u + \alpha q) = \Delta u + \alpha$, we can shift the right-hand side by a constant by adding the corresponding multiple of $q$.  Our choice of $2$ in this manuscript makes the scaling limit of \eref{fde} have a particularly nice structure, and which makes the rigorous determination of the scaling limit cleaner than it would be to confirm Sportiello's ansatz for the case where the cutoff is 3 \cite{Sportiello}.

Note that for the standard $\leq 3$ cutoff, our solutions $u_n$ correspond via $\Delta u_n+1$ to the unique recurrent configuration equivalent to the all-ones configuration on $\Z^2\cap n\Omega$, whereas the identity element is the unique recurrent configuration equivalent to the all-zeros configuration.

\section{Pattern Stability}

In this section we prove our main result, the stability of patterns.  A translation of the odometer $o$ is any function of the form
\begin{equation*}
  \tilde o(x) = o(x + y) + z \cdot x + w
\end{equation*}
for some $y,z \in \Z^2$ and $w \in \Z$.  Note that $\tilde o$ also satisfies \pref{structure}.  In particular, we have the following.
\begin{lemma}
  \label{l.translations}
  For any odometer $o$ and translation $\tilde o$, we have
  \begin{equation}
  \tilde o(x)=o(x)+b\cdot x+r(x),
  \end{equation}
  for $b\in \Z^2$ and $r:\Z^2\to \Z$ is a $V\Z^3$ periodic function.\qed
\end{lemma}

Throughout the remainder of this section, we fix choices of $P,A,V,T,o$ from \pref{structure}.  The following theorem says that, when a recurrent function is close to $o$, then it is equal to translations $\tilde o$ of $o$ in balls covering almost the whole domain.  This is pattern stability.

\begin{theorem}
  \label{t.stability}
  There is a universal constant $C_{\ref{t.stability}} > 0$ such that if $h\geq C_{\ref{t.stability}}$, $r\geq C_{\ref{t.stability}}|V|$, $hr\geq C_{\ref{t.stability}}|V|^3$, $R \geq C_{\ref{t.stability}} h r$, and $v : \Z^2 \to \Z$ is recurrent and satisfies $|v - o| \leq h^2$ in $B_R$, then, for a \new{$(1 - C_{\ref{t.stability}} R^{-1} rh)$}-fraction of points $x$ in $B_{R - r}$, there is a translation $\tilde o_x$ of $o$ such that $v = \tilde o_x$ in $B_r(x)$.
\end{theorem}

We introduce the norms
\begin{equation*}
  |x|_V = |V^t x|_\infty
\end{equation*}
and
\begin{equation*}
  |x|_{V^{-1}} = \min \{ |y|_1 : V y = x \}.
\end{equation*}
We prove these norms are dual and comparable to Euclidean distance.

\begin{lemma}
  \label{l.norms}
  For all $x, y \in \Z^2$, we have
  \begin{equation*}
    |x \cdot y| \leq |x|_V |y|_{V^{-1}}
  \end{equation*}
  and
  \begin{equation*}
    C_{\ref{p.structure}}^{-1} |x| \leq |V| |x|_{V^{-1}} \leq C_{\ref{p.structure}} |x|.
  \end{equation*}
\end{lemma}

\begin{proof}
  If $f : \R^2 \to \R^3$ satisfies $V f(x) = x$ and $| x |_{V^{-1}} = |f(x)|_1$, then
  \begin{equation*}
    |x \cdot y| = |x \cdot V f(y)| = |V^t x \cdot f(y)| \leq |V^t x|_\infty |f(y)|_1 = |x|_V |y|_{V^{-1}}.
  \end{equation*}
  The latter two inequalities follow from $1 \leq |V|^2 \leq C_{\ref{p.structure}} \det(V)$.
\end{proof}

The ``web of twos'' provided by \pref{structure} allows us to show that, when two recurrent function differs from $o$, then the difference must grow.  This is a quantitative form of the maximum principle.

\begin{lemma}
  \label{l.separation}
  If $v : \Z^2 \to \Z$ is recurrent, $x_0 \sim y_0 \in \Z^2$, $v(x_0) = o(x_0)$, and $v(y_0) \neq o(y_0)$, then, for all $k \geq 0$,
  \begin{equation*}
    \max_{|x - x_0|_{V^{-1}} \leq k + 1} (o - v)(x) \geq k.
  \end{equation*}
\end{lemma}

\begin{proof}
  We inductively construct $T_k = T + V z_k$ such that
  \begin{enumerate}
  \item $x_0, y_0 \in T_0$,
  \item $o - v$ is not constant on $T_k$,
  \item $|z_{k+1} - z_k|_1 \leq 1$,
  \item $\max_{T_{k+1}} (o - v) > \max_{T_k} (o - v)$.
  \end{enumerate}
  Since $|x - x_0|_{V^{-1}} \leq 1$ for all $x \in T_0$, this implies the lemma.

  The base case is immediate from the fact that every lattice edge is contained in a single tile.  For the induction step, we use the recurrence of $o$ and $v$.  In particular, since $T_k \subseteq \Z^2$ is finite, the difference $o - v$ attains its extremal values in $T_k$ on the boundary $\partial T_k$.  Since $o - v$ is not constant on $T_k$, it is not constant on $\partial T_k$.  Therefore, we may select $x \in \partial T_k$ such that
  \begin{equation*}
    \max_{T_k} (o - v) = (o - v)(x)
  \end{equation*}
  and
  \begin{equation*}
    (o - v)(x) > (o - v)(y) \mbox{ for some } y \sim x.
  \end{equation*}  
  Now, if $(o-v)(x) \geq (o-v)(y)$ for all $y \sim x$, then, using from \pref{structure} that $\Delta o(x) = 2$, we compute
  \begin{equation*}
    - 1 \geq \Delta (o - v)(x) = \Delta o(x) - \Delta v(x) = 2 - \Delta v(x),
  \end{equation*}
  contradicting the recurrence of $v$.  Thus we can find $y \sim x$ such that $(o - v)(x) < (o - v)(y)$.  Choose $z_{k+1} \in \Z^3$ such that $|z_{k+1} - z_k|_1 \leq 1$ and $x, y \in T_{k+1} = T + V z_{k+1}$.
\end{proof}

On the other hand, we can approximate any linear separation of odometers, showing that the above lemma is nearly optimal.

\begin{lemma}
  \label{l.translation}
For any $b \in \R^2$, there is a translation $\tilde o$ of $o$ such that
  \begin{equation*}
    |\tilde o(x) - o(x) - b \cdot x| \leq \tfrac23 |x|_{V^{-1}} + 2 C_{\ref{p.structure}}|V|^2 \quad \mbox{for } x \in \Z^2
  \end{equation*}
\end{lemma}

\begin{proof}
  For $y, z \in \Z^2$ to be determined, let
  \begin{equation*}
    \tilde o(x) = o(x + y) + z \cdot x - o(y) + \tilde o(0).
  \end{equation*}
  Using the quadratic polynomial $q$ from \pref{structure}, compute
  \begin{equation*}
    \begin{aligned}
      |\tilde o(x) - o(x) - b \cdot x| - 2 C_{\ref{p.structure}} |V|^2
      & \leq |q(x + y) + z \cdot x - q(y) + q(0) - q(x) - b \cdot x| \\
      & = |(P y + z - b) \cdot x | \\
      & \leq |P y + z - b|_V |x|_{V^{-1}}.
    \end{aligned}
  \end{equation*}
  We claim that we can choose $y, z \in \Z^2$ such that
  \begin{equation*}
    |P y + z - b|_V \leq \tfrac23.
  \end{equation*}
  Indeed, using \pref{structure}, we compute
  \begin{equation*}
    \begin{aligned}
    \{ V^t (P y + z) : y, z \in \Z^2 \}
    & = \{ A^t y + V^t z: y, z \in \Z^2 \} \\
    & \supseteq \{ (A^t Q V + V^t Q A) w : w \in \Z^3 \} \\
    & \supseteq \{ Q' w : w \in \Z^3 \} \\
    & = \{ w \in \Z^3 : \smat{1 \\ 1 \\ 1} \cdot w = 0 \}.
    \end{aligned}
  \end{equation*}
  Since $V^t b \in \{ w \in \R^3 : \smat{1 \\ 1 \\ 1} \cdot w = 0 \}$, the result follows.
\end{proof}

We prove the main ingredient of \tref{stability} by combining the previous two lemmas.  Recall that a function $u : \Z^2 \to \R$ touches another function $v : \Z^2 \to \R$ from below in a set $X \subseteq \Z^2$ at the point $x \in X$ if $\min_X (v - u) = (v-u)(x) = 0$.

\begin{lemma}
  \label{l.touch}
  There is a universal constant $C_{\ref{l.touch}} > 1$ such that, if
  \begin{enumerate}
  \item $R \geq C_{\ref{l.touch}} |V|^3$,
  \item $v : \Z^2 \to \Z$ is recurrent in $B_R$,
  \item $\psi(x) = o(x) - \tfrac12 |V|^2 R^{-2} |x - y|^2 + k$ for some $k\in \Z$,
  \item $\psi$ touches $v$ from below at $0$ in $B_R$,
  \end{enumerate}
  then there is a translation $\tilde o$ of $o$ such that
  \begin{equation*}
    v = \tilde o \quad \mbox{in } B_{C_{\ref{l.touch}}^{-1} R}
  \end{equation*}
\end{lemma}

\begin{proof}
  Using \lref{translation}, we may choose a translation $\tilde o$ of $o$ such that
  \begin{equation*}
    |\psi(x) + \tfrac12 |V|^2 R^{-2} |x|^2 - \tilde o(x)| \leq \tfrac23 |x|_{V^{-1}} + 2 C_{\ref{p.structure}} |V|^2
    \quad \mbox{for } x \in B_R.
  \end{equation*}
  Next, since $\psi$ touches $v$ from below at $0$, we obtain
  \begin{equation}\label{ovlower}
    \tilde o(0) - v(0) \geq -2 C_{\ref{p.structure}} |V|^2
  \end{equation}
  and
  \begin{equation}
    \label{ovupper}
    \tilde o(x) - v(x) \leq (2 C_{\ref{p.structure}}+1)|V|^2 + \tfrac23 |x|_{V^{-1}} \quad \mbox{in } B_R.
  \end{equation}
  We show that, if $C_{\ref{l.touch}} > 0$ is a sufficiently large universal constant, then $\tilde o - v$ is constant in $B_{C_{\ref{l.touch}}^{-1} R}$.  Suppose not.  Then there are $x \sim y \in B_{C_{\ref{l.touch}}^{-1} R}$ with
  \begin{equation*}
    (\tilde o - v)(0) = (\tilde o - v)(x) \neq (\tilde o - v)(y).
  \end{equation*}
  Since $x\in B_{C_{\ref{l.touch}}^{-1}R}$ we have from \lref{norms} that
  \begin{equation*}
    |x|_{V^{-1}} \leq C_{\ref{p.structure}}C_{\ref{l.touch}}^{-1} |V|^{-1} R.
  \end{equation*}
     By \lref{norms}, $B_R$ contains all points $z$ with $|z-x|_{V^{-1}}\leq R|V|^{-1}(C_{\ref{p.structure}}^{-1}-C_{\ref{p.structure}}C_{\ref{l.touch}}^{-1})$.  In particular, by \lref{separation}, there is a $z\in B_R$ such that
  \begin{equation}\label{ovfromtiles}
    (\tilde o - v)(z) \geq - C_{\ref{p.structure}}|V|^2 + R|V|^{-1}(C_{\ref{p.structure}}^{-1}-C_{\ref{p.structure}}C_{\ref{l.touch}}^{-1})-1.
  \end{equation}
  Combining \eqref{ovfromtiles} with \eqref{ovupper}, we obtain
  \begin{equation}
R|V|^{-1}(C_{\ref{p.structure}}^{-1}-C_{\ref{p.structure}}C_{\ref{l.touch}}^{-1})- C_{\ref{p.structure}}|V|^2 -1\leq (C_{\ref{p.structure}}+1)|V|^2 + \tfrac23 |x|_{V^{-1}},
  \end{equation}
which is impossible for $R\geq C_{\ref{l.touch}} |V|^3$ and $C_{\ref{l.touch}}$ large relative to $C_{\ref{p.structure}}$.
     
\end{proof}

We prove pattern stability by adapting the growth lemma for non-divergence form elliptic equations, see for example \cite{Savin}.  The above lemma is used to show that the ``touching map'' is almost injective.

\begin{proof}[Proof of \tref{stability}]
  Our proof assumes
  \[
   hr\geq C_{\ref{l.touch}}|V|^3,\quad h\geq 2C_{\ref{l.touch}},\quad r\geq 3 |V|,\quad R\geq 3hr+2C_{\ref{l.touch}}r
   \]

  Step 1.
  We construct a touching map.  For $y \in B_{R - 4 h r}$, consider the test function
  \begin{equation*}
    \varphi_y(x) = o(x) - \tfrac12 |V|^2 r^{-2} |x - y|^2.
  \end{equation*}
  Observe that
  \begin{equation*}
    (v - \varphi_y)(y) = (v - o)(y) \leq h^2
  \end{equation*}
  and, for $z \in B_R \setminus B_{3hr}(y)$,
  \begin{equation*}
    (v - \varphi_y)(z) = (v - o)(z) + \tfrac12 |V|^2 r^{-2} |z - y|^2 \geq - h^2 + \tfrac92 |V|^2 h^2 \geq \tfrac{7}{2} h^2.
  \end{equation*}
  We see that $v - \varphi_y$ attains its minimum over $B_R$ at some point $x_y \in B_{3hr}(y)$.  Assuming $hr \geq C_{\ref{l.touch}}|V|^3$ and $h\geq 2C_{\ref{l.touch}}$, and $R\geq 3hr+2C_{\ref{l.touch}}r$, we have that $R-3hr\geq 2C_{\ref{l.touch}}r$, and \lref{touch} gives a translation $o_y$ of $o$ such that
  \begin{equation*}
    v = o_y \quad \mbox{in } B_{2r}(x_y).
  \end{equation*}
  The map $y \mapsto x_y$ is the touching map.

  Step 2.  We know that $v$ matched a translation of $o$ in a small ball around every point in the range of the touching map.  If we knew the touching map was injective, then the fraction of these good points would be $|B_{R-4hr}| / |B_R|$.  While injectivity generally fails, we are able to show almost injectivity, in the following sense.
  
  \textbf{Claim:} For every $y \in B_{R - 4 hr}$, there are sets $y \in T_y \subseteq B_R$ and $x_y \in S_y \subseteq B(x_y,|V|)$ such that $|T_y| \leq |S_y|$ and $S_y \cap S_{\tilde y} \neq \emptyset$ implies $S_y = S_{\tilde y}$ and $T_y = T_{\tilde y}$.

  To prove this, observe first that, assuming $r \geq 3|V|$,
  \begin{equation*}
    |x_y - x_{\tilde y}| \leq 2 |V| \mbox{ implies } o_y = o_{\tilde y},
  \end{equation*}
  since $B_{2r}(x_y)\cap B_{2r}(x_{\tilde y})$ contains four $V\Z^3$-equivalent (not collinear) points, which is sufficient to determine an odometer translation uniquely.
  
  Next, observe from \lref{translations} that for every $y_0 \in B_{R \setminus 4 hr}$ there is a slope $b \in \Z^2$ such that, for all $y, x \in \Z^2$,
  \begin{equation}\label{zy0y}
    o_{y_0}(x) - \varphi_{y}(x) = r(x) + b \cdot x + \tfrac12 |V|^2 r^{-2}|x-y|^2.
  \end{equation}
Now let $z_{y_0,y} = \operatorname{argmin} (o_{y_0} - \varphi_{y})$, and let $\mathcal X$ be any tiling of $\Z^2$ by the $V \Z^3$-translations of a fundamental domain with diameter bounded by $|V|$.  We define
  \begin{align*}
    S_{y_0,y}&=X\in \mathcal T\quad\text{such that}\quad z_{y_0,y}\in X,\\
    T_{y_0,y}&= \{ \tilde y : z_{y_0,\tilde y} \in S_{y_0,y} \}.
  \end{align*}
  Note that $S_{y_0,y}$ and $T_{y_0,y}$ depend only on $o_y$ and $\tilde y$ (and not directly on $y$).  Moreover, since $y \mapsto z_{y_0, y}$ commutes with $V \Z^3$-translation by \eqref{zy0y}, 
we have that no $T_{y_0,y}$ can contain two $V\Z^3$-equivalent points, and thus that $|T_{y_0,y}|\leq |S_{y_0,y}|=|V|$ for all $y_0,y$. 
  
  Finally, since $|x_{y_0} - x_{y_1}| \leq 2 |V|$ implies $o_{y_0} = o_{y_1}$, we see that $|x_{y_0} - x_{y_1}| \leq 2 |V|$ implies $S_{y_0,y} = S_{y_1,y}$ and $T_{y_0,y} = T_{y_1, y}$.  Letting $S_y = S_{y,y}$ and $T_y = T_{y,y}$, we see that $S_y \cap S_{\tilde y} \neq \emptyset$ implies $|y - \tilde y| \leq 2 |V|$ and thus $S_{y} = S_{\tilde y}$ and $T_{y} = T_{\tilde y}$, as required.
 
  Step 3. Let $\mathcal Y \subseteq B_{R - 4 hr}$ be maximal subject to $\{ S_y : y \in \mathcal Y \}$ being disjoint.  By the implication in the claim, we must have $B_{R - 4 hr } \subseteq \cup \{ T_y : y \in \mathcal Y \}$.  We compute
  \begin{equation*}
    | \cup_{y \in \mathcal Y} S_y | \geq \sum_{y \in \mathcal Y} |T_y| \geq \left| \cup_{y \in \mathcal Y} T_y \right| \geq |B_{R - 4 hr}|.
  \end{equation*}
  Finally, observe that at each point of $x \in \cup \{ S_y : y \in \mathcal Y \}$, there is a translation $\tilde o_x$ of $o$ such that $v = \tilde o_x$ in $B_r(x)$.  The theorem now follows from the estimate \new{$|B_{R - 4 hr}| / |B_R| \geq 1 - C R^{-1} r h$}.
\end{proof}

\section{Explicit Solution}

In this section, we describe the solution of
\begin{equation}
  \label{e.pdesquare}
  \begin{cases}
    D^2 \bar u \in \partial \Gamma & \mbox{in } \new{(0,1)^2} \\
    \bar u = 0 & \mbox{on } \R^2 \setminus \new{(0,1)^2}.
  \end{cases}
\end{equation}
As one might expect from \fref{identity}, the solution is piecewise quadratic and satisfies the stronger constraint $D^2 \bar u \in \Gamma^+$.  The algorithm described here is implemented in the code attached to the arXiv upload.

\begin{theorem}[\cite{Levine-Pegden-Smart-1,Levine-Pegden-Smart-2}]
  \label{t.piecewise}
  There are disjoint open sets $\Omega_k \subseteq \Omega = (0,1)^2$ and constants $L > 1$, \new{ $\delta > 0$} such that the following hold.
  \begin{enumerate}
  \item $\sum_k |\Omega_k| = |\Omega|$ and $|\Omega_k| \leq \new{ k^{-\delta}}$.
  \item $D^2 \bar u$ is constant in each $\Omega_k$ with value $P_k \in \Gamma^+$.
  \item The $V_k \in \Z^{2 \times 3}$ corresponding to $P_k$ via \pref{structure} satisfies $|V_k| \leq L^k$.
  \item For $r > 0$, $| \{ x : B_r(x) \subseteq \Omega_n \} | \geq |\Omega_n| - L |\Omega_n|^{1/2} r$.
  \end{enumerate}
\end{theorem}

Since this result is essentially contained in \cite{Levine-Pegden-Smart-1} and \cite{Levine-Pegden-Smart-2}, we omit the proofs, giving only the explicit construction and a reminder of its properties.  We construct a family of super-solutions $\bar v_n \in C(\R^2)$ of \eref{pdesquare} such that $\bar v_n \downarrow \bar u$ uniformly as $n \to \infty$.  Each super-solution $\bar v_n$ is a piecewise quadratic function with finitely many pieces.  The measure of the pieces whose Hessians do not lie in $\Gamma^+$ goes to zero as $n \to \infty$.  The Laplacians of the first eight super-solutions is displayed in \fref{supersolutions}.  The construction is similar to that of a Sierpinski gasket, and the pieces are generated by an iterated function system.

\begin{figure}[h]
  \includegraphics[angle=90,width=.23\textwidth]{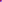}
  \hfill
  \includegraphics[angle=90,width=.23\textwidth]{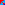}
  \hfill
  \includegraphics[angle=90,width=.23\textwidth]{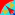}
  \hfill
  \includegraphics[angle=90,width=.23\textwidth]{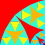} \\
  \bigskip
  \includegraphics[angle=90,width=.23\textwidth]{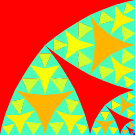}
  \hfill
  \includegraphics[angle=90,width=.23\textwidth]{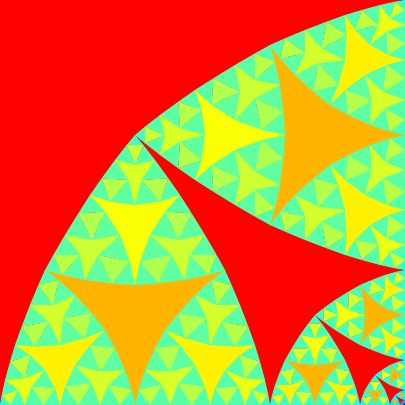}
  \hfill
  \includegraphics[angle=90,width=.23\textwidth]{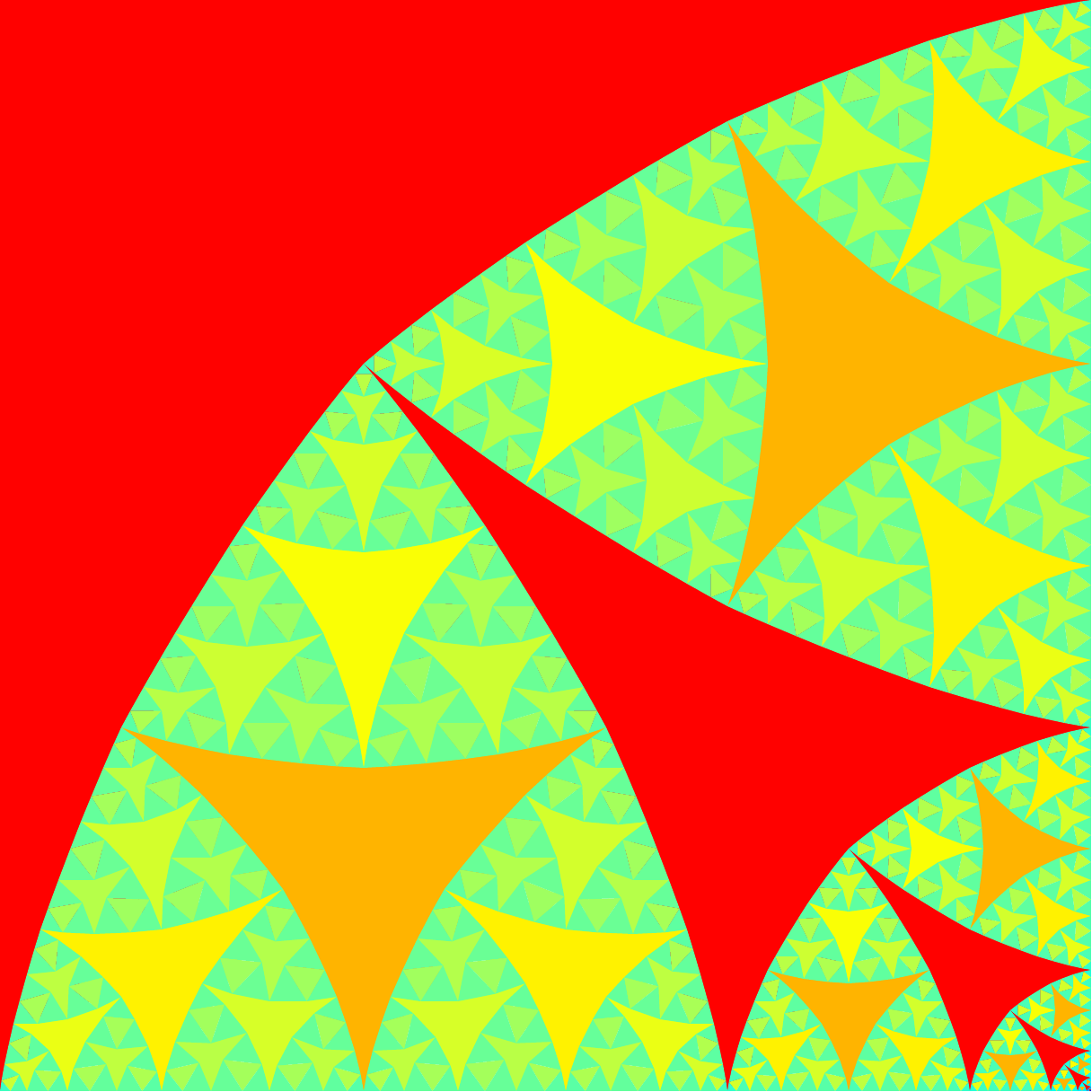}
  \hfill
  \includegraphics[angle=90,width=.23\textwidth]{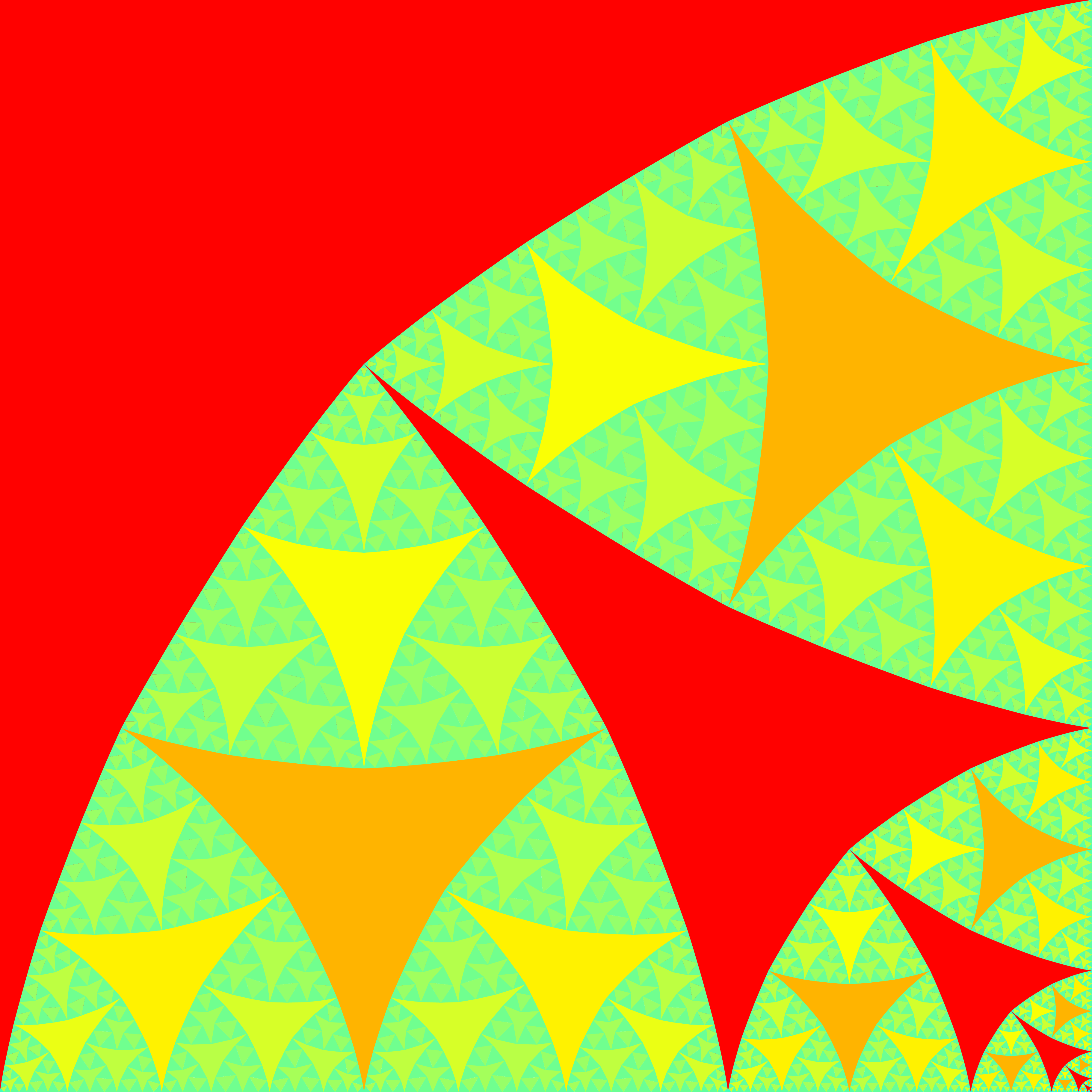}
  \caption{The Laplacian of the supersolution $\bar v_n$ on $(0,1)^2$ for $n = 0, ..., 7$}
  \label{f.supersolutions}
\end{figure}

The solution is derived from the following data.

\begin{definition}
    Let $z_s, a_s, w_s, b_s \in \C^3$ for $s \in \{ 1,2,3 \}^{< \omega}$ satisfy
  \begin{equation*}
    z_{()} = \mat{1 \\ 1+i \\ i},
    \ 
    a_{()} = \mat{0 \\ -1 \\ i},
    \ 
    z_{s k} = Q R^k z_s,
    \
    a_{s k} = \overline Q R^k z_s,
    \
    w_s = S z_s,
    \mbox{ and }
    b_s = \overline S a_s,
  \end{equation*}
  where
  \begin{equation*}
    Q = \frac13 \mat{3 & 0 & 0 \\ 1 + i & 1 - i & 1 \\ 1 - i & 1 & 1 + i},
    \ 
    R = \mat{1 & 0 & 0 \\ 0 & 0 & 1 \\ 0 & 1 & 0},
    \mbox{ and }
    S = \frac13 \mat{1 & 1+i & 1-i \\ 1-i & 1 & 1+i \\ 1+i & 1-i & 1}.
  \end{equation*}
\end{definition}

The above iterated function system generates four families of triangles, which we use to define linear maps by interpolation.

\begin{definition}
  For $z, a \in \C^3$, let $L_{z,a}$ be the linear interpolation of the map $z_k \mapsto a_k$.  That is, $L_{z,a}$ has domain
  \begin{equation*}
    \triangle_z = \{ t_1 z_1 + t_2 z_2 + t_3 z_3 : t_1, t_2, t_3 \geq 0 \mbox{ and } t_1 + t_2 + t_3 = 1 \}
  \end{equation*}
  and satisfies
  \begin{equation*}
    L_{z,a}(t_1 z_1 + t_2 z_2 + t_3 z_3) = t_1 a_1 + t_2 a_2 + t_3 a_3.
  \end{equation*}
\end{definition}

Identifying $\C$ and $\R^2$ in the usual way, $L_{z,a}$ is a map between triangles in $\R^2$.  We glue together the linear maps $L_{z_s,a_s}$ and $L_{w_s,b_s}$ to construct the gradients of our super-solutions.  The complication is that the triangles $\triangle_{z_s}$ and $\triangle_{w_s}$ are not disjoint.  As we see in \fref{supersolutions}, the domains of later maps intersect the earlier ones.  We simply allow the later maps to overwrite the earlier ones.

\begin{definition}
  For integers $n \geq k \geq -1$, let $G_{n,k} : (0,1)^2 \to \R^2$ satisfy
  \begin{equation*}
    G_{n,-1}(x) = \smat{0 & 0 \\ 0 & 1}x,
  \end{equation*}
  \begin{equation*}
    G_{n,n}(x) = \begin{cases}
      L_{z_s,a_s}(x) & \mbox{if } x \in \triangle_{z_s} \mbox{ for } s \in \{ 1, 2, 3 \}^n \\
      G_{n,n-1}(x) & \mbox{otherwise},
    \end{cases}
  \end{equation*}
  and, for $n > k > -1$,
  \begin{equation*}
    G_{n,k}(x) = \begin{cases}
      L_{w_s,b_s}(x) & \mbox{if } x \in \triangle_{w_s} \mbox{ for } s \in \{ 1, 2, 3 \}^k \\
      G_{n,k-1}(x) & \mbox{otherwise}.
    \end{cases}
  \end{equation*}
  For integers $n \geq 0$, let $G_n = G_{n,n}$
\end{definition}

It follows by induction that $G_n$ is continuous and the gradient of a supersolution:

\begin{proposition}[\cite{Levine-Pegden-Smart-1}]
  For $n \geq 0$, there is a $\bar v_n \in C(\R^2) \cap C^1((0,1)^2)$ such that
  \begin{equation*}
    \bar v_n(x_1,x_2) = \bar v_n(|x_1|,|x_2|),
  \end{equation*}
  \begin{equation*}
    \bar v_n = 0 \quad \mbox{on } \R^2 \setminus (\new{ 0},1)^2,
  \end{equation*}
  and
  \begin{equation*}
    D \bar v_n(x) = G_n(x) + \smat{1 & 0 \\ 0 & 0} x \quad \mbox{for } x \in (0,1)^2
  \end{equation*}
  Moreover, $\sup_n \sup_{\new{(0,1)^2}} |D\bar v_n| < \infty$.
\end{proposition}

The above proposition implies that the gradients $D L_{z_s, a_s}$ and $D L_{w_s, b_s}$ are symmetric matrices.  In fact, one can prove that
\begin{equation*}
  D L_{z_s,a_s} + \smat{ 1 & 0 \\ 0 & 0 } \in \Gamma
  \quad \mbox{and} \quad
  D L_{w_s,b_s} + \smat{ 1 & 0 \\ 0 & 0 } \in \Gamma^+.
\end{equation*}
Passing to the limit $n \to \infty$, one obtains \tref{piecewise}.

\begin{remark}
  Observe that the intersection points of the pieces of the explicit solution all have triadic rational coordinates.  We expect this is connected to \qref{perfect}.
\end{remark}

\section{Quantitative Convergence}

In order to use \tref{stability} to prove appearance of patterns, we need a rate of convergence.  Throughout this section, fix a bounded convex set $\Omega \subseteq \R^2$ and functions $u_n : \Z^2 \to \Z$ and $\bar u \in C(\R^2)$ that solve \eref{fde} and \eref{pde}, respectively.  We know that rescalings $\bar u_n(x) = n^{-2} u_n(n x) \to \bar u(x)$ uniformly in $x \in \R^2$ as $n \to \infty$.  We quantify this convergence using the additional regularity afforded by \tref{piecewise}.  The additional regularity arrives in the form of local approximation by recurrent functions.

\begin{definition}
  We say that $\bar u$ is $\ep$-approximated if $\ep \in (0,1/2)$ and there is a constant $K \geq 1$ such that the following holds for all $n \geq 1$:  For a $1 - K n^{-\ep}$ fraction of points $x \in \Z^2 \cap n \Omega$, there is a $u : \Z^2 \to \Z$ that is recurrent in $B_{n^{1-\ep}}(x) \subseteq n \Omega$ and satisfies $\max_{y \in B_{n^{1-\ep}}(x)} |u(y) - n^2 \bar u(n^{-1} y)| \leq K n^{2-3\ep}$.
\end{definition}

Being $\ep$-approximated implies quantitative convergence of $\bar u_n$ to $\bar u$.

\begin{theorem}
  \label{t.convergence}
  If $\bar u$ is $\ep$-approximated, then there is an $L > 0$ such that
  \begin{equation*}
    \sup_{x \in \Z^2} |u_n(x) - n^2 \bar u(n^{-1} x)| \leq L n^{2 - \ep/8}
  \end{equation*}
  holds for all $n \geq 1$.
\end{theorem}

A key ingredient of our proof of this theorem is a standard ``doubling the variables'' result from viscosity solution theory.  This is analogous to ideas used in the convergence result of \cite{Caffarelli-Souganidis} for monotone difference approximations of fully nonlinear uniformly elliptic equations.  In place of $\delta$-viscosity solutions, we use \cite[Lemma 6.1]{Armstrong-Smart} as a natural quantification of the Theorem on Sums \cite{Crandall} in the uniformly elliptic setting.  In the following lemma, we abuse notation and use the Laplacian both for functions on the rescaled lattice $n^{-1} \Z^2$ and the continuum $\R^2$.

\begin{lemma}
  \label{l.doubling}
  Suppose that
  \begin{enumerate}
  \item $\Omega \subseteq \R^2$ is open, bounded, and convex,
  \item $u : n^{-1} \Z^2 \to \R$ satisfies $|\Delta u| \leq 1$ in $n^{-1} \Z^2 \cap \Omega$ and $u = 0$ in $n^{-1} \Z^2 \setminus \Omega$,
  \item $v \in C(\R^2)$ satisfies $|\Delta v| \leq 1$ in $\Omega$ and $v = 0$ in $\R^2 \setminus \Omega$,
  \item $\max_{n^{-1} \Z^2} (u - v) = \ep > 0$.
  \end{enumerate}
  There is a $\delta > 0$ depending only on $\Omega$ such that, for all $p, q \in B_{\delta \ep}$, the function
  \begin{equation*}
    \Phi(x,y) = u(x) - v(y) - \delta \ep |x^2| - \delta^{-1} \ep^{-1} |x - y|^2 - p \cdot x - q \cdot y
  \end{equation*}
  attains its maximum over $n^{-1} \Z^2 \times \R^2$ at a point $(x^*, y^*)$ such that $B_{\delta \ep}(x^*) \subseteq \Omega$ and $B_{\delta \ep}(y^*) \subseteq \Omega$.  Moreover, the set of possible maxima $(x^*,y^*)$ as the slopes $(p,q)$ vary covers a $\delta^{10} \ep^8$ fraction of $(n^{-1} \Z^2 \cap \Omega) \times \Omega$.
\end{lemma}

\begin{proof}
  Step 1.
  Standard estimates for functions with bounded Laplacian (both discrete and continuous) imply that $\bar u_n$ and $\bar u$ are Lipschitz with a constant depending only on the convex set $\Omega$.  Estimate
  \begin{equation*}
    \bar u_n(x) - \bar u(x) \geq \Phi(x,x) \geq \bar u_n(x) - \bar u(x) - C \delta \ep
  \end{equation*}
  and, using the Lipschitz estimates,
  \begin{equation*}
    \Phi(x,y) \leq \Phi(x,x) + C |x - y| - \delta^{-1} \ep^{-1} |x-y|^2.
  \end{equation*}
  Thus, if $\max_{x,y} \Phi(x,y) = \Phi(x^*,y^*)$, then
  \begin{equation*}
    \ep - C \delta \ep \leq \Phi(x^*,y^*) \leq \ep + C \delta \ep - C \delta^{-1} \ep^{-1} |x^* - y^*|^2.
  \end{equation*}
  In particular, if $\delta > 0$ is sufficiently small, then
  \begin{equation*}
    |x^* - y^*| \leq C \delta \ep.
  \end{equation*}
  Using the boundary conditions in combination with the Lipschitz estimates, we see that, provided $\delta > 0$ is sufficiently small, $B_{\delta \ep}(x^*) \subseteq \Omega$ and $B_{\delta \ep}(y^*) \subseteq \Omega.$

  Step 2.  The final measure-theoretic statement is an immediate consequence of the fact that the touching map $(p,q) \mapsto (x^*,y^*)$ has a $\delta^{-5} \ep^{-4}$-Lipschitz inverse.  This is a consequence of the proof of \cite[Lemma 6.1]{Armstrong-Smart}.  Here, one must substitute the discrete Alexandroff-Bakelman-Pucci inequality \cite{Lawler,Kuo-Trudinger} since we have the discrete Laplacian.  The statement we obtain is that, if $\delta > 0$ is sufficiently small, $(p_i,q_i) \mapsto (x_i,y_i)$, and $|(x_1,y_1) - (x_2,y_2)| \leq \delta^2 \ep$, then $|(p_1,q_1) - (p_2,q_2)| \leq \delta \ep$.  The result now follows by a covering argument.
\end{proof}

\begin{proof}[Proof of \tref{convergence}]
  Suppose $\bar u$ is $\ep$-approximated and $K \geq 1$ is the corresponding constant.  For $L > 1$ to be determined, suppose for contradiction that
  \begin{equation*}
    \max_{n^{-1} \Z^2} (\bar u_n - \bar u) \geq L n^{-\ep/8}.
  \end{equation*}
  (The case of the other inequality is symmetric.)  Apply \lref{doubling} with $u = \bar u_n$, $v = \bar u$, and $\ep = L n^{-\ep/8}$.  As the slopes $(p,q)$ vary, the maximum $(x^*,y^*)$ of $\Phi$ satisfies $B_{\delta L n^{-\ep/8}}(y^*) \subseteq \Omega$ and the set of possible $y^*$ covers a $L^8 \delta^{10} n^{-\ep}$ fraction of $n^{-1} \Z^2 \cap \Omega$.  Thus, if $L > 1$ is large enough, we may choose $(p,q)$ such that there is a function $w : \Z^2 \to \Z$ that is recurrent in $B_{n^{1-\ep}}$ and satisfies $\max_{y \in B_{n^{1-\ep}}} |w(y) - n^2 \bar u(n y^* + y)| \leq K n^{2-3\ep}$.  

  Consider
  \begin{multline*}
    z \mapsto \Phi(x^* + z, y^* + z) - \Phi(x^*,y^*) = \\ \bar u_n(x^* + z) - \bar u_n(x^*) - \bar u(y^* + z) - \bar u(y^*) - (2 \delta \ep x^* + p + q) \cdot z - \delta \ep |z|^2,
  \end{multline*}
  which attains its maximum at $0$.  Let $r \in \Z^2$ denote the integer rounding of $2 \delta \ep x^* + p + q$.  Observe that
  \begin{equation*}
    (2 \delta \ep x^* + p + q) \cdot x + \delta \ep |z|^2 \geq (K+1) n^{2 - 3 \ep} \quad \mbox{for } |z| \geq n^{1 - \ep},
  \end{equation*}
  provided that $L > 1$ is large enough.  In particular,
  \begin{equation*}
    z \mapsto \bar u_n(z^* + z) - w(z) - r \cdot z
  \end{equation*}
  attains a strict local maximum in $B_{n^{1-\ep}}$.  This contradicts the maximum principle for recurrent functions.
\end{proof}

\section{Convergence of Patterns}

We prove \tref{main} by combining \tref{stability}, \tref{convergence}, and \tref{piecewise}.

\begin{proof}
  First observe that \tref{piecewise} implies that $\bar u$ is $\alpha$-approximated for some $\alpha > 0$.  Making $\alpha > 0$ smaller, \tref{convergence} implies
  \begin{equation*}
    \sup |\bar u_n - \bar u| \leq C n^{2 - \alpha}
  \end{equation*}
  from \tref{convergence}.  Let us now consider what happens inside an individual piece $\Omega_k$.  For $1 < r < R < n$, \tref{piecewise} implies that a
  \begin{equation*}
    1 - C \tau^{-k/2} n^{-1} R
  \end{equation*}
  fraction of points $x$ in $n \Omega_k$ satisfy $B_R(x) \subseteq n \Omega_k$.  There is an odometer $o_k$ for $P_k$ such that
  \begin{equation*}
    h^2 = \max_{B_R(x)} |u_n - o_k| \leq C n^{2 - \alpha} + C R.
  \end{equation*}
  \old{
  Assuming $r \geq C L^k \geq C |V_k|$, \tref{stability} implies that $\Delta o_k$ $r$-matches $\Delta u_n$ at a
  \begin{equation*}
    1 - C R^{-2} r h
  \end{equation*}
  fraction of points in $B_R(x)$.  Assuming that $r \leq n^{\alpha}$ and setting
  \begin{equation*}
    R = n^{1 - \alpha/3} \tau^{k/6} r^{1/3},
  \end{equation*}
  these together imply that $u_n$ $r$-matches $\Delta o_k$ at a
  \begin{equation*}
    1 - C \tau^{-k/3} n^{-\alpha/3} r^{1/3}
  \end{equation*}
  fraction of points in $n \Omega_k$.  Replacing $C \tau^{-k/3}$ by a larger $L^k$, we can remove the restrictions on $r$, as the estimate becomes trivial at the edges.
}
  
  \bigskip
  \new{
    Assuming $r \geq C L^k \geq C |V_k|$, \tref{stability} implies that $\Delta o_k$ $r$-matches $\Delta u_n$ at a
    \begin{equation*}
    (1 - C R^{-1} r h)
    \end{equation*}
    fraction of points in $B_R(x)$.  In particular, $u_n$ $r$-matches $\Delta o_k$ at at least a 
    \begin{equation*}
      (1 - C \tau^{-k/2} n^{-1} R)^2(1 - C R^{-1} r h)\geq 1-C(\tau^{-k/2}n^{-1} R+R^{-1} r h)
    \end{equation*}
    fraction of points. 

   In particular, assuming $r\leq n^{\alpha/2}$ and setting
  \begin{equation*}
    R = n^{1 - \alpha/4} \tau^{k/4} r^{1/2},
  \end{equation*}
  these together imply that $u_n$ $r$-matches $\Delta o_k$ at a
  \begin{equation*}
    1 - C \tau^{-k/4} n^{-\alpha/4} r^{1/2}
  \end{equation*}
  fraction of points in $n \Omega_k$.  Replacing $C \tau^{-k/4}$ by a larger $L^k$, since in the regime $r>n^{\alpha}$ the claim is vacuous.}
\end{proof}


\begin{bibdiv}
  \begin{biblist}


    \bib{Caracciolo-Paoletti-Sportiello}{article}{
        author={Sergio Caracciolo},
        author={Guglielmo Paoletti},
        author={Andrea Sportiello},
        title={Conservation laws for strings in the Abelian Sandpile Model},
        journal={Europhysics Letters},
        volume={90},
        number={6},
        year={2010},
        pages={60003},
	note={\arxiv{1002.3974}}       
    }
\bib{dhar2009pattern}{article}{
  title={Pattern formation in growing sandpiles},
  author={Dhar, Deepak}, 
  author={Sadhu, Tridib},
  author={Chandra, Samarth},
  journal={Europhysics Letters},
  volume={85},
  number={4},
  pages={48002},
  year={2009},
  publisher={IOP Publishing}
  note={\arxiv{0808.1732}}
}

\bib{sadhu2010pattern}{article}{
  title={Pattern formation in growing sandpiles with multiple sources or sinks},
  author={Sadhu, Tridib},
  author={Dhar, Deepak},
  journal={Journal of Statistical Physics},
  volume={138},
  number={4-5},
  pages={815--837},
  year={2010},
  publisher={Springer},
  note={\arxiv{0909.3192}}
}

    \bib{Armstrong-Smart}{article}{
      author={Armstrong, Scott N.},
      author={Smart, Charles K.},
      title={Quantitative stochastic homogenization of elliptic equations in nondivergence form},
      journal={Arch. Ration. Mech. Anal.},
      volume={214},
      date={2014},
      number={3},
      pages={867--911},
      issn={0003-9527},
      review={\MR{3269637}},
      doi={10.1007/s00205-014-0765-6},
    }

    \bib{Caffarelli-Souganidis}{article}{
      author={Caffarelli, Luis A.},
      author={Souganidis, Panagiotis E.},
      title={A rate of convergence for monotone finite difference approximations to fully nonlinear, uniformly elliptic PDEs},
      journal={Comm. Pure Appl. Math.},
      volume={61},
      date={2008},
      number={1},
      pages={1--17},
      issn={0010-3640},
      review={\MR{2361302}},
      doi={10.1002/cpa.20208},
    }

    \bib{Crandall}{article}{
      author={Crandall, Michael G.},
      title={Viscosity solutions: a primer},
      conference={
        title={Viscosity solutions and applications},
        address={Montecatini Terme},
        date={1995},
      },
      book={
        series={Lecture Notes in Math.},
        volume={1660},
        publisher={Springer, Berlin},
      },
      date={1997},
      pages={1--43},
      review={\MR{1462699}},
      doi={10.1007/BFb0094294},
    }

    \bib{Dhar-Sadhu-Chandra}{article}{
      title={Pattern formation in growing sandpiles},
      author={Dhar, Deepak}, 
      author={Sadhu, Tridib},
      author={Chandra, Samarth},
      journal={Europhysics Letters},
      volume={85},
      number={4},
      pages={48002},
      year={2009},
      publisher={IOP Publishing},
      note={\arxiv{0808.1732}}
    }
    
    \bib{Kalinin-Shkolnikov}{article}{
      author={Kalinin, Nikita},
      author={Shkolnikov, Mikhail},
      title={Tropical curves in sandpiles},
      note={Preprint (2015) \arxiv{1509.02303}}
    }

    \bib{Kuo-Trudinger}{article}{
      author={Kuo, Hung-Ju},
      author={Trudinger, Neil S.},
      title={A note on the discrete Aleksandrov-Bakelman maximum principle},
      booktitle={Proceedings of 1999 International Conference on Nonlinear
        Analysis (Taipei)},
      journal={Taiwanese J. Math.},
      volume={4},
      date={2000},
      number={1},
      pages={55--64},
      issn={1027-5487},
      review={\MR{1757983}},
    }

    \bib{Lawler}{article}{
      author={Lawler, Gregory F.},
      title={Weak convergence of a random walk in a random environment},
      journal={Comm. Math. Phys.},
      volume={87},
      date={1982/83},
      number={1},
      pages={81--87},
      issn={0010-3616},
      review={\MR{680649}},
    }
    
    \bib{Levine-Pegden-Smart-1}{article}{
      author={Levine, Lionel},
      author={Pegden, Wesley},
      author={Smart, Charles K},
      title={Apollonian structure in the Abelian sandpile},
      note={Preprint (2012) \arxiv{1208.4839}}
    }
    
    \bib{Levine-Pegden-Smart-2}{article}{
      author={Levine, Lionel},
      author={Pegden, Wesley},
      author={Smart, Charles K},
      title={The Apollonian structure of integer superharmonic matrices},
      note={Preprint (2013) \arxiv{1309.3267}}
    }
    
    \bib{Levine-Propp}{article}{
      author={Levine, Lionel},
      author={Propp, James},
      title={What is $\dots$ a sandpile?},
      journal={Notices Amer. Math. Soc.},
      volume={57},
      date={2010},
      number={8},
      pages={976--979},
      issn={0002-9920},
      review={\MR{2667495}},
    }

    \bib{Ostojic}{article}{
      title={Patterns formed by addition of grains to only one site of an abelian sandpile},
      author={Ostojic, Srdjan},
      journal={Physica A: Statistical Mechanics and its Applications},
      volume={318},
      number={1},
      pages={187--199},
      year={2003},
      publisher={Elsevier}
    }
    
    \bib{Pegden-Smart}{article}{
      author={Pegden, Wesley},
      author={Smart, Charles K},
      title={Convergence of the Abelian Sandpile},
      journal={Duke Mathematical Journal, to appear},
      note={\arxiv{1105.0111}}
    }

    \bib{Savin}{article}{
      author={Savin, Ovidiu},
      title={Small perturbation solutions for elliptic equations},
      journal={Comm. Partial Differential Equations},
      volume={32},
      date={2007},
      number={4-6},
      pages={557--578},
      issn={0360-5302},
      review={\MR{2334822}},
      doi={10.1080/03605300500394405},
    }

    \bib{Sportiello}{article}{
      author={Sportiello, Andreas},
      title={The limit shape of the Abelian Sandpile identity},
      note={Limit shapes, ICERM 2015}
    }
    
  \end{biblist}
\end{bibdiv}
\end{document}